\documentclass[a4paper,11pt]{amsart}
\usepackage{amsmath,amssymb}
\usepackage[utf8]{inputenc}
\usepackage{fancyhdr}
\usepackage{verbatim}
\usepackage[pagewise]{lineno} %displaymath, mathlines

\newcommand*\patchAmsMathEnvironmentForLineno[1]{%
  \expandafter\let\csname old#1\expandafter\endcsname\csname #1\endcsname
  \expandafter\let\csname oldend#1\expandafter\endcsname\csname
  end#1\endcsname
  \renewenvironment{#1}%
     {\linenomath\csname old#1\endcsname}%
     {\csname oldend#1\endcsname\endlinenomath}}% 
\newcommand*\patchBothAmsMathEnvironmentsForLineno[1]{%
  \patchAmsMathEnvironmentForLineno{#1}%
  \patchAmsMathEnvironmentForLineno{#1*}}%
\AtBeginDocument{%
\patchBothAmsMathEnvironmentsForLineno{equation}%
\patchBothAmsMathEnvironmentsForLineno{align}%
\patchBothAmsMathEnvironmentsForLineno{flalign}%
\patchBothAmsMathEnvironmentsForLineno{alignat}%
\patchBothAmsMathEnvironmentsForLineno{gather}%
\patchBothAmsMathEnvironmentsForLineno{multline}%
}
\usepackage{enumitem}
\setenumerate[1]{label=(\roman{*}), ref=(\roman{*})}
\setenumerate[2]{label=(\alph{*}), ref=(\alph{*})}

\newcommand\enn{\mathbb N}
\newcommand\err{\mathbb R}

\newcommand{\eps}{\varepsilon}

\newcommand{\HB}{\text{H{\kern -0.35em}B}}
\newcommand{\vertiii}[1]{{\left\vert\kern-0.25ex\left\vert\kern-0.25ex\left\vert#1\right\vert\kern-0.25ex\right\vert\kern-0.25ex\right\vert}}

\DeclareMathOperator{\cspann}{\overline{span}}

\title{Two properties of M\"{u}ntz spaces}

\author[T.~A.~Abrahamsen]{Trond A.~Abrahamsen}
\address{Department of Mathematics, University of
  Agder, Postboks 422, 4604 Kristiansand, Norway.}
\email{trond.a.abrahamsen@uia.no}
\urladdr{http://home.uia.no/trondaa/index.php3}
\author[A.~Leraand]{Aleksander Leraand}
\email{aleksander9109@gmail.com}
\author[A.~Martiny]{Andr{\'e} Martiny}
\email{andremartiny91@gmail.com}
\author[O.~Nygaard]{Olav Nygaard}
\email{Olav.Nygaard@uia.no}
\urladdr{http://home.uia.no/olavn/}

\keywords {M\"{u}ntz space; Asymptotically isometric copy of $c_0$;
  Octahedral space; Diameter 2 properties} 

\subjclass[2010]{46E15, 46B04, 46B20, 26A99}

\newtheorem{thm}{Theorem}[section]

\newtheorem{lem}[thm]{Lemma}
\newtheorem{cor}[thm]{Corollary}

\theoremstyle{definition}

\newtheorem{defn}[thm]{Definition}

\theoremstyle{remark}

\newtheorem{rem}[thm]{Remark}

\begin{document}
%\nocite{*}
%\linenumbers

\begin{abstract}
  We show that M\"{u}ntz spaces, as subspaces of $C[0,1]$, contain
  asymptotically isometric copies of $c_0$ and that their dual spaces are octahedral.  
\end{abstract}

\maketitle

\section{Introduction}\label{sec:intro}

  Let $\Lambda = (\lambda_k)_{k=0}^\infty$ be a strictly increasing sequence of
  non-negative real numbers and let $M(\Lambda) =
  \cspann\{t^{\lambda_k}\}_{k=0}^\infty \subset C[0,1]$ where $C[0,1]$
  is the space of real valued continuous functions on $[0,1]$ endowed with the
  $\max$-norm. We will call $M(\Lambda)$ a M{\"u}ntz space provided
  $\sum_{k=1}^\infty 1/\lambda_k < \infty$. The name is
  justified by M\"{u}ntz' wonderful discovery that if $\lambda_0
  = 0$ then $M(\Lambda) = C[0,1]$ if and only if $\sum_{k=1}^\infty
  1/\lambda_k=\infty$. 

  It is well known that $C[0,1]$ contains isometric copies of $c_0$ (see
  e.g. \cite[p.~86]{AlKal} how to construct them) and that its dual
  space is isometric to an $L_1(\mu)$ space for some measure $\mu$. The
  aim of this paper is to demonstrate that M\"{u}ntz spaces inherit quite a bit of structure
  from $C[0,1]$ in that they always contain
  asymptotically isometric copies of $c_0$, and that their dual spaces are always
  octahedral. (An $L_1(\mu)$ space is octahedral. See below for an
  argument.) Let us proceed by recalling the definitions of these two
  concepts and put them into some context.

\begin{defn}{\cite[Theorem~2]{MR1606342}}\label{def:ascnull} 
 A Banach space $X$ is said to contain an \emph{asymptotically isometric
 copy of $c_0$} if there exist a sequence $(x_n)_{n=1}^\infty$ in $X$ and constants
 $0 < m < M < \infty$ such that for all sequences 
 $(t_n)_{n=1}^\infty$ with finitely many non zero terms 
 \begin{align*}
   m \sup_n |t_n|
		\leq \left\| \sum_n t_n x_n \right\|
		\leq M \sup_n |t_n|,
 \end{align*}
  and
 \begin{align*}
  \lim_{n\rightarrow \infty} \|x_n\| = M. 
 \end{align*}
\end{defn}

  R.~C. James proved a long time ago (see
  \cite{James}) that $X$ contains an almost isometric copy of $c_0$ as
  soon at is contains a copy of $c_0$. Note that containing an
  asymptotically isometric copy of $c_0$ is a stronger property, see
  e.g. \cite[Example~5]{MR1606342}. 

\begin{defn}\label{def:octa} A Banach space $X$ is said to be
  \emph{octahedral} if for any finite-dimensional subspace $F$ of $X$ and
  every $\varepsilon>0$, there exists $y\in S_X$ with \[\|x + y\|\geq
  (1-\varepsilon)(\|x\| + 1)\:\:\mbox{for all}\:\:x\in F.\]
\end{defn}

  This concept was introduced by Godefroy in
  \cite{G-octa}, and there also the following result can be found on
  page 12 (see also \cite{HLP} for a proof of it):

\begin{thm}\label{thm:gd-coco}
  Let $X$ be a Banach space. Then $X^*$ is octahedral if and only if
  every finite convex combination of slices of $B_X$ has diameter 2.
\end{thm}

  By a slice of $B_X$ we mean a set of the form \[S(x^*,\eps):=\{x \in B_X:
  x^*(x) > 1 - \eps, \eps >0, x^* \in S_{X^*}\}.\]

  When we show that the dual of M\"{u}ntz spaces are octahedral we
  will use Theorem \ref{thm:gd-coco} and establish the equivalent
  property stated there. Note that an $L_1(\mu)$ space is
  octahedral. Indeed, the bidual of such a space can be written
  $L_1(\mu)^{**} = L_1(\mu) \oplus_1 X$  for some subspace $X$ of
  $L_1(\mu)^{**}$ (see e.g. \cite[IV.~Example~1.1]{HWW}). From here the
  octahedrality of $L_1(\mu)$ is a straightforward application of the
  Principle of Local Reflexivity.

  We do not know of much research in the direction of our results. But
  we would like to mention a paper of P. Petr\'{a}\v{c}ek (\cite{Petr}),
  where he demonstrates that M\"{u}ntz spaces are never reflexive and
  asks whether they can have the Radon-Nikod\'{y}m property. Since the
  Radon-Nikod\'{y}m property implies the existence of slices of
  arbitrarily small diameter,  we now understand that M\"{u}ntz spaces
  rather belong to the ``opposite world'' of Banach spaces.

\section{Results}
\begin{defn}
  We will say that a strictly increasing sequence of non-negative real
  numbers $(\lambda_k)_{k=0}^\infty$ has the \emph{Rapid Increase
    Property (RIP) } if $\lambda_{k+1} \geq 2\lambda_k$ for every $k \ge 0$. 
  
  We will call a function of the form
  \begin{align*}
    p(x) = x^\alpha - x^\beta,
  \end{align*}
  where $0 \le \alpha < \beta,$ a \emph{spike function}. 
\end{defn}

\begin{rem}\label{rem:spike}
  If $\alpha > 0$ it should be clear that any spike function $p$
  satisfies $p(0) = p(1) = 0,$ attains its norm on a unique point
  $x_p$, is strictly increasing on $[0, x_p],$ and strictly decreasing on $[x_p, 1].$
\end{rem}

We will need the following result below.
	
\begin{lem}\label{thm:lemmaB}
  Let $(\lambda_k)_{k=0}^\infty$ be an RIP sequence and
  $(p_k)_{k=0}^\infty$ the sequence of corresponding spike
  functions $p_k(x) = x^{\lambda_k} - x^{\lambda_{k+1}}$. Then $\inf_k
  \|p_k\| \ge 1/4.$ Moreover, the sequence
  $(p_k/\|p_k\|)_{k=1}^\infty$ converges to $0$ weakly in $M(\Lambda)$.
\end{lem}
	
\begin{proof}
  We want to find the norm of the spike function defined by
  \begin{align*}
     p_k(x)= x^{\lambda_{k}} - x^{\lambda_{k+1}}.
  \end{align*}
   Observe that $r_k(x):= x^{\lambda_k}-x^{2\lambda_k} \le p_k(x)$ for
   all $x\in[0,1]$. Now, by standard calculus, $r_k$ attains its
   maximum at $x_k$ where ${x_k}^{\lambda_k}=\frac{1}{2}$. Thus
   \[\|p_k\|\geq r_k(x_k) = \frac{1}{2}-(\frac{1}{2})^2=\frac{1}{4}.\] 

   As $(p_k)_{k=1}^\infty$ converges pointwise to $0$ and
   $\inf_k\|p_k\| \ge 1/4,$ the sequence $(p_k/\|p_k\|)_{k=1}^\infty$ converges
   pointwise to $0$ and thus weakly to $0$ as it is bounded. 
\end{proof}	

\begin{rem}\label{rem:nap}
  By standard calculus one can show that the point at which $p_k$ in
  Lemma \ref{thm:lemmaB}  obtains its norm is $\bar{x}_k = 
  (\lambda_k/\lambda_{k+1})^{1/(\lambda_{k+1}- \lambda_k)}$. For
  sufficiently large $\lambda_k$ (e.g. for $\lambda_k > 3$) it is
  straightforward to show that
  \[\bar{x}_k \ge 1/( \lambda_{k+1} -\lambda_k )^{1/(\lambda_{k+1}
      -\lambda_k)} = : y_k.\] $(y_k)$ is strictly monotone, and converges to 1.
 
\end{rem}
	
\begin{thm}
 The dual of any Müntz space is octahedral.
\end{thm}

\begin{proof}
  Let $M(\Lambda)$ be a M{\"u}ntz space. Let
  \begin{align*}
    C = \sum_{j=1}^n \mu_j S(x_j^*, \eps_j),
  \end{align*}
  where $\sum_{j=1}^n \mu_j = 1, \mu_j > 0$, and $S(x_j^*,
  \eps_j), 1\leq j\leq n$, is a slice of $B_{M(\Lambda)}.$
  We will show that the diameter of $C$ is 2 (cf. Theorem
  \ref{thm:gd-coco}). To this end, start with some $f \in C$ and write
  $f = \sum_{j=1}^n \mu_j g^j $, where $g^j \in S(x_j^*, \eps_j)$.  Let
  $(\lambda_k)_{k=0}^\infty$ be an RIP subsequence of $\Lambda$ and put
  \begin{align*}
    h_k^{j+} &= g^j + (1 - g^j(x_k))\frac{p_k}{\|p_k\|}\\
    h_k^{j-} &= g^j - (1 + g^j(x_k))\frac{p_k}{\|p_k\|}
  \end{align*}
  where $(p_k)_{k=0}^\infty$ is the sequence of spike functions
  corresponding to $(\lambda_k)_{k=0}^\infty$ and $x_k$ the (unique)
  point where $p_k$ attains its norm.  We will prove that, for any
  $\varepsilon>0$, there  exists a $K=K(\varepsilon)$ such that
  whenever $k\geq K$ we have $\frac{1}{1 + \eps} h_k^{j+}, \frac{1}{1
    + \eps} h_k^{j-} \in  S(x_j^*, \eps_j)$ for every $1 \le j \le n$. Then,
  clearly
    \begin{align*}
      \frac{1}{1 + \eps} \sum_{j=1}^n \mu_j h_k^{j\pm} \in C,     
    \end{align*}
    and 
     \begin{align*}
       \left\| \frac{1}{1 + \eps} \sum_{j=1}^n \mu_j h_k^{j+} -\frac{1}{1 +
       \eps} \sum_{j=1}^n \mu_j h_k^{j-} \right\| \ge \frac{1}{1 +
       \eps}\left(\sum_{j=1}^n \mu_j
       [h_k^{j+}(x_k)-h_k^{j+}(x_k)]\right) = \frac{2}{1 + \eps}. 
    \end{align*}
    for all $k \geq K.$ Since $\eps$ is arbitrary, we can thus conclude
    that $C$ has diameter 2.
 
  To produce the $K = K(\varepsilon)$ above, note that $h_k^{j\pm}$ converges
  to $g^j$ pointwise, and thus weakly since the sequences are
  bounded. As $U_j:=\{x\in M(\Lambda):
  x_j^\ast(x)>1-\varepsilon\}$ is weakly open, each sequence
  $(h_k^{j\pm})_{k=0}^\infty$ enters $U_j$ eventually. Since there are
  only a finite number of sets $U_j$, this entrance is uniform. So, what
  is left to prove is that for $\varepsilon>0$ there exists $K$ such
  that $\|h_k^{j\pm}\| \leq 1+\varepsilon$ whenever $k\geq K$.
 
  Now, let $\eps > 0$. Combining Remark \ref{rem:spike}, Remark
  \ref{rem:nap}, that $(p_k/\|p_k\|)_{k=1}^\infty$ converges pointwise
  to $0$, and the continuity of $g^j$, we can find $K \in \enn$ such
  that for all $k \geq K$ there are points $0 < a_k <x_k < b_k < 1$ such that
  \begin{align*}
    \frac{p_k(x)}{\|p_k\|} > \eps &\Leftrightarrow x
                \in (a_k, b_k), \\\sup_{u, v \in (a_k, b_k)} |g^j(u)
                - g^j(v)| &< \eps, \hskip 1mm j=1, \ldots, n.              
  \end{align*}
  We will see that this $K$ does the job for the given
  $\varepsilon>0$: Let $k \geq K$ and suppose $x \not\in (a_k, b_k)$.
  Then
  \begin{align*}
   |h_k^{j+}(x)| = \left| g^j(x) + (1 - g^j(x_k))
    \frac{p_k(x)}{\|p_k\|} \right| \leq |g^j(x)| + 2 \eps \le 1 + 2\eps.
  \end{align*}
  If $x \in (a_k, b_k)$, observe that 
  \begin{align*}
    |h^{j+}(x)| &\le \left| g^j(x) + (1-
    g^j(x))\frac{p_k(x)}{\|p_k\|}\right| + |g^j(x) -
                  g^j(x_k)|\frac{p_k(x)}{\|p_k\|}\\
   & <   \left| g^j(x) + (1- g^j(x))\frac{p_k(x)}{\|p_k\|}\right| + \eps.
  \end{align*}
   Now, if $g^j(x) \geq 0,$ then
  \begin{align*}
   \left| g^j(x) + (1 - g^j(x))\frac{p_k(x)}{\|p_k\|} \right| \le
    g^j(x) + (1 - g^j(x)) = 1.
  \end{align*}
   If $g^j(x) < 0$ and $g^j(x) + (1 - g^j(x))p_k(x)/\|p_k\| \ge 0,$ then
   \begin{align*}
    \left| g^j(x) + (1 - g^j(x))\frac{p_k(x)}{\|p_k\|} \right| \le
     g^j(x) + (1 - g^j(x)) = 1.
   \end{align*}
   If $g^j(x) < 0$ and $g^j(x) + (1 - g^j(x))p_k(x)/\|p_k\| < 0,$ then
   \begin{align*}
    \left| g^j(x) + (1 -  g^j(x))\frac{p_j(x)}{\|p_k\|} \right| \le
     |g^j(x)| \le 1.
    \end{align*}
   In any case we have for $k \geq K$ and $x \in [0, 1]$ that $|h_k^{j+}(x)| \le 1 + 2
   \eps.$  The argument that $\|h_k^{j-}\| \leq 1+\varepsilon$ is similar.
\end{proof}

\begin{thm}
  M{\"u}ntz spaces contain asymptotically isometric copies of $c_0$.
\end{thm}
	
\begin{proof}
   We will construct a sequence $(f_n)_{n=1}^\infty \subset
   M(\Lambda)$ and pairwise disjoint intervals $I_n = (a_n, b_n) \subset
   [0,1]$ such that for all $n \in \enn$ 
   \begin{itemize}

      \item[(i)]$f_n (x) \geq 0$ for all $x \in [0,1],$
      \item[(ii)] $\|f_n\| = 1 - 1/2^n,$
      \item[(iii)] $b_n < a_{n+1}$
      \item[(iv)] $f_n(x) > 1/2^{2n} \Leftrightarrow x \in I_n,$
      \item[(v)] $f_n(x) < 1/2^{2m}$ whenever $m \ge n$ and $x \in I_m.$ 
    \end{itemize}
		
   To this end choose a subsequence of $\Lambda$
   with the RIP (which is possible as $\sum_{k=1}^\infty 1/\lambda_k <
   \infty$). For simplicity denote also this subsequence by
   $(\lambda_k)_{k=0}^\infty$. Let $(p_k)_{k=1}^\infty$ be its
   corresponding sequence of spike functions, and let $x_k$ be the
   (unique) point in $(0,1)$ where $p_k$ obtains its maximum. 
   
   Now, start by letting $k_1 = 1$ and put
   \begin{align*}
     f_1 = (1 - 1/2) \frac{p_{k_1}}{\|p_{k_1}\|}.
   \end{align*}
   Using continuity and properties of $p_1$, we can find an interval
   $I_1 = (a_1, b_1)$ such that $0 < a_1 < b_1 < 1$ and $f_1(x) >
   \frac{1}{2^2} \Leftrightarrow x \in I_1.$ By construction $f_1$ satisfies
   the conditions (i) - (iv). 

   To construct $f_2$ we use Lemma \ref{thm:lemmaB} and Remarks
   \ref{rem:spike} and \ref{rem:nap} to find
   $k_2 \in \enn$ and an interval $I_2= (a_2, b_2)$ with $b_1 < a_2 < b_2 < 1$ such that
   \begin{align*}
     x \in I_2 \Leftrightarrow p_{k_2}(x) &> \frac{1/2^4}{1 - 1/2^2}\|p_{k_2}\|, \\
     x \in I_2 \Rightarrow p_{k_1}(x) & \le \frac{1}{2^4}.
   \end{align*}
   Let
   \begin{align*}
     f_2 = (1 - 1/2^2)\frac{p_{k_2}}{\|p_{k_2}\|}.
   \end{align*}
   By construction $f_1$ now satisfies condition (v) for $m \le 2$ and $f_2$ satisfies
   conditions (i) - (iv). 

   To construct $f_3$ we use Lemma \ref{thm:lemmaB} and Remarks
   \ref{rem:spike} and \ref{rem:nap} again to find
   $k_3 \in \enn$ and an interval $I_3 = (a_3, b_3)$ with $b_2 < a_3 < b_3 < 1$ such that
   \begin{align*}
     x \in I_3 \Leftrightarrow p_{k_3}(x) &> \frac{1/2^6}{1 - 1/2^3}\|p_{k_3}\|, \\
     x \in I_3 \Rightarrow p_{k_j}(x) & \le \frac{1}{2^6}
                                               \hskip 2mm \text{ for }
                                               j = 1, 2.
   \end{align*}
   Let
   \begin{align*}
     f_3 = (1 - 1/2^3)\frac{p_{k_3}}{\|p_{k_3}\|}.
   \end{align*}
   By construction $f_1$ and $f_2$ now satisfy condition (v) for $m
   \le 3$ and $f_3$ satisfies conditions (i) - (iv). If we continue in
   the same manner we obtain a sequence $(f_n)_{n=1}^\infty \subset
   M(\Lambda)$ and a sequence of intervals $I_n = (a_n, b_n)$ which
   satisfies the conditions (i) - (v). 
   
   Now we will show that $(f_n)_{n=1}^\infty$
   satisfies the requirements of Definition \ref{def:ascnull}. To this
   end we need
   to find constants $0 < m < M < \infty$ such that given any sequence
   $(t_n)_{n=1}^\infty$ with finitely many non zero terms
   \begin{equation}\label{DLT-criteria1}
     m \sup_n |t_n|
    \leq \| \sum_n t_n f_n \|
    \leq M \sup_n |t_n|
   \end{equation}
   and
   \begin{equation}\label{DLT-criteria2}
     \lim_{n \rightarrow \infty} \|f_n\| = M
   \end{equation}
   We claim that \eqref{DLT-criteria1} and \eqref{DLT-criteria2} holds
   with $m = \frac{1}{4}$ and $M = 1.$ First observe that we have
   $\lim_{n \rightarrow \infty} \|f_n\| =
   1$ immediately from the requirements, so \eqref{DLT-criteria2}
   holds for $M = 1$. In order to prove the
   two inequalities in \eqref{DLT-criteria1}, let $(t_n)_{n=1}^\infty$ be an
   arbitrary sequence with finitely many non zero terms. First we will prove that $1/4
   \sup_n |t_n| \leq \left\| \sum_n t_n f_n \right\|$. We can assume
   by scaling that $\sup |t_n| = 1$. Since $(t_n)_{n=1}^\infty$ has
   finitely many non zero terms, its norm is attained at, say, $n = N$, i.e.
   $|t_N| = 1$. Put $x_N = x_{k_N}$ where $x_{k_N}$ is the point where
   $p_{k_N}$ and thus $f_N$ attains its norm. Then
   \begin{align*}
      \| \sum_{n \in \enn} t_n f_n \| & \geq | t_Nf_N(x_N)|  -
                                        |\sum_{n \not= N} t_n f_n(x_N)|\\
      &\ge 1 - \frac{1}{2^N} - \sum_{n \not=
                                        N} |f_n(x_N)| \\&> 1-
        \frac{1}{2^N} - \frac{1}{4} \ge \frac{1}{4}.
   \end{align*}
   We conclude that the left hand side of the inequality
   \eqref{DLT-criteria1} holds. Now we will
   show the right hand side of this inequality holds, i.e. we
   want to prove that $|\sum_n t_n f_n(x)| \leq 1$ for all $x \in
   [0,1]$. Since $f_n \ge 0$ for all $n=1, 2, \ldots,$ we may assume that
   every $t_n$ is positive. Now, if $x \not \in \cup_n (a_n, b_n)$, we have
   \begin{align*}
     \sum_n t_n f_n(x)
		\leq \sum_n f_n(x)
		\leq \sum_n \frac{1}{2^{2n}}
		\leq \frac{1}{3}.
   \end{align*}
   If, on the other hand $x \in (a_{n'}, b_{n'})$ for some $n' \in \enn,$ then
   \begin{align*}
     \sum_n t_n f_n(x)
		&\leq f_{n'}(x) + \sum_{n < n'} f_n(x) + \sum_{n > n'} f_n(x)\\
		&\leq 1 - \frac{1}{2^{n'}} + \frac{n' - 1}{2^{2n'}} + \frac{1}{2^{2n'}}\\
		& \le 1 + \frac{n' - 2^{n'}}{2^{2n'}} \le 1 -
                  \frac{1}{2^{2n'}} < 1.      
   \end{align*}
   These combined yields the right hand side of the inequality
   \eqref{DLT-criteria1}, so the proof is complete.
\end{proof}

  A Banach space $X$ contains an \emph{asymptotically isometric copy of
  $\ell_1$} if it contains a sequence  $(x_n)_{n=1}^\infty$ for which
  there exists a sequence $(\delta_n)_{n=1}^\infty$ in $(0,1)$,
  decreasing to $0$, and such that
  
  \begin{equation*}
    %\label{eq:2}
     \sum_{n=1}^m(1 - \delta_n)|a_n| \le \|\sum_{n=1}^m a_n x_n\| \le
    \sum_{n=1}^m |a_n|
  \end{equation*}
   for each finite sequence $(a_n )_{n=1}^m$ in $\err$.  
   
  Merging (\cite[Theorem~2]{DJLT}) and \cite[Lemma 2.3]{ALNT} gives us
  that if either the Banach space $X$ contains an asymptotically
  isometric copy of $c_0$ or if $X^{\ast}$ is octahedral, then
  $X^{\ast}$ contains an asymptotically isometric copy of
  $\ell_1$. So, we have two ways of proving

\begin{cor}\label{kor:aiell1}
     $M(\Lambda)^{\ast}$ contains an asymptotically isometric copy of $\ell_1$.
\end{cor}

  Moreover, we have

\begin{cor}
    $M(\Lambda)^{\ast\ast}$ contains an isometrically isomorphic copy of $L_1[0,1]$.
\end{cor}

\begin{proof}
   This follows from Corollary~\ref{kor:aiell1} and \cite[Theorem~2]{MR1751149}.
\end{proof}

\end{document}